\documentclass[12pt]{amsart}
\usepackage[margin=1.2in]{geometry}

\usepackage{comment}

\usepackage{amsmath,amsfonts,amssymb,amstext,amscd,amsthm}
\usepackage{mathtools}
\usepackage{hyperref}
\usepackage{graphicx}
\usepackage{color}

\newcommand{\R}{\ensuremath{\mathbb{R}}}

\DeclareMathOperator{\sech}{\mathrm{sech}} 
 
\DeclareMathOperator{\area}{\mathrm{area}}

\def\labelitemi{--}
\def\ba #1\ea {\begin{align} #1\end{align}}
\def\bann #1\eann {\begin{align*} #1\end{align*}}
\def\ben #1\een {\begin{enumerate} #1\end{enumerate}}
\def\bi #1\ei {\begin{itemize}\renewcommand\labelitemi{--} #1\end{itemize}}

\theoremstyle{plain}
\newtheorem{thm}{Theorem}[section]
\newtheorem*{thm*}{Theorem}
\newtheorem{lem}[thm]{Lemma}
\newtheorem{cor}[thm]{Corollary}

\newtheorem{prop}[thm]{Proposition}

\theoremstyle{remark}
\newtheorem{rem}{Remark}%[section]
\makeatletter
\newcommand{\pushright}[1]{\ifmeasuring@#1\else\omit\hfill$\displaystyle#1$\fi\ignorespaces}
\newcommand{\pushleft}[1]{\ifmeasuring@#1\else\omit$\displaystyle#1$\hfill\fi\ignorespaces}
\makeatother

\title[Ancient CSF in the disc with mixed boundary condition]{Ancient curve shortening flow in the disc with mixed boundary condition}

\date{\today}

\author{Mat Langford}\author{Yuxing Liu}\author{George McNamara}
\address{Mathematical Sciences Institute, Australian National University, Canberra, ACT, Australia}
\email{mathew.langford@anu.edu.au}
\email{yuxing.liu@anu.edu.au}
\email{george.mcnamara@anu.edu.au}
\begin{document}

\begin{abstract}
Given any non-central interior point $o$ of the unit disc $D$, the diameter $L$ through $o$ is the union of two linear arcs emanating from $o$ which meet $\partial D$ orthogonally, the shorter of them stable and the longer unstable (under these boundary conditions). In each of the two half discs bounded by $L$, we construct a convex eternal solution to curve shortening flow which fixes $o$ and meets $\partial D$ orthogonally, and evolves out of the unstable critical arc at $t=-\infty$ and into the stable one at $t=+\infty$. We then prove that these two (congruent) solutions are the only non-flat convex ancient solutions to the curve shortening flow satisfying the specified boundary conditions. We obtain analogous conclusions in the ``degenerate'' case $o\in\partial D$ as well, although in this case the solution contracts to the point $o$ at a finite time with asymptotic shape that of a half Grim Reaper, thus providing an interesting example for which an embedded flow develops a collapsing singularity.

%We also obtain an interesting auxiliary result: in the non-degenerate case, any convex non-flat initial curve with monotone curvature flows for all positive time, converging to the stable critical arc as time tends to infinity. In the degenerate case, any such curve contracts to the endpoint $o$ after a finite amount of time, with asymptotic shape that of a half Grim Reaper.
\end{abstract}

\maketitle

\section{Introduction}

Variational problems subject to boundary constraints are ubiquitous in pure and applied mathematics and physics. One of the simplest such problems is to find and study paths of critical (e.g. minimal) length amongst those joining a given point $o$ in some domain $\Omega$ to its boundary $\partial\Omega$. When $\Omega$ is a Euclidean domain, such paths are, of course, straight linear arcs from $o$ to $\partial\Omega$ which meet $\partial\Omega$ orthogonally. 

While characterizing all such curves is a non-trivial problem in general (even for convex Euclidean domains, say), the ``Dirichlet--Neumann geodesics'' in the unit disc in $\R^2$ are easily found: when $o$ is the origin, they are the radii; when $o$ is not the origin, there are exactly two, and their union is the diameter through $o$.

One useful tool for analyzing such variational problems is the (formal) gradient flow (a.k.a. steepest descent flow), which in this case is the ``Dirichlet--Neuman curve shortening flow''; this equation evolves each point of a given curve with velocity equal to the curvature vector at that point, subject to holding one endpoint fixed at $o$ with the other constrained to $\partial\Omega$, which is met orthogonally. 

While curve shortening flow is now well-studied under other boundary conditions --- particularly the ``periodic'' (i.e. no-boundary) \cite{MR2967056,MR2843240,AndrewsBryan,BLTcsf,DHScsf,GageHamilton86,Gage84,Grayson,Huisken96}, ``Neumann--Neumann'' (a.k.a. free boundary) \cite{MR4668092,Buckland,Edelen,Huisken89,Ko,FBDistanceComparison,Stahl96b,Stahl96a} and ``Dirichlet--Dirichlet'' \cite{ALT,Huisken96} conditions --- we are aware of no literature considering the mixed ``Dirichlet--Neumann'' condition. 

Our main result (inspired by \cite{MR4668092}) is the following classification of the convex ancient solutions which arise in the simple setting of the unit disc.

\begin{thm}\label{thm:convex ancient solutions}
Given any $d\in(0,1]$, there exists a convex%\footnote{Herein, an arc $\Gamma$ in a convex planar domain $\Omega$ is called \emph{convex} if the union of $\Gamma$ with the distance minimizing curves joining its endpoints to $\partial\Omega$ bounds a relatively convex region in $\Omega$.}
, locally uniformly convex ancient solution $\{\Gamma_t^d\}_{t\in(-\infty,\omega_d)}$ to curve shortening flow in the unit disc $D$ with one endpoint fixed at $o:=(-d,0)$ and the other meeting $\partial D$ orthogonally. The timeslices $\Gamma^d_t$ each lie in the upper half-disc, and converge uniformly in the smooth topology as $t\to-\infty$ to the unstable critical arc $\{(x,0):x\in [-d,1]\}$; as a graph over the $x$-axis, %$\{\Gamma_t^d\}_{t\in(-\infty,\omega_d)}$ satisfies
\[
\mathrm{e}^{\lambda^2t}y(x,t)\to A\sinh(\lambda(x+d))\;\;\text{uniformly in $x$ as}\;\; t\to-\infty
\]
for some $A>0$, where $\lambda$ is the positive solution to $\tanh(\lambda(1+d))=\lambda$.

When $d<1$, $\omega_d=+\infty$ and the timeslices converge uniformly in the smooth topology as $t\to+\infty$ to the minimizing arc $\{(x,0):x\in [-1,-d]\}$. When $d=1$, $\omega_d<\infty$ and the timeslices contract uniformly as $t\to \omega_d$ to the point $o$ and, after performing a standard type-II blow-up, converge locally uniformly in the smooth topology to the right half of the downward translating Grim Reaper. 

Modulo time translations and reflection about the $x$-axis, $\{\Gamma_t^d\}_{t\in(-\infty,\omega_d)}$ is the only non-flat convex ancient curve shortening flow subject to the same boundary conditions.
\end{thm}

En route to proving Theorem \ref{thm:convex ancient solutions}, we establish the following convergence result (cf. \cite{ALT,GageHamilton86,Gage84,Grayson,FBDistanceComparison}), which is of independent interest (see the proof of Lemma \ref{lem:old solutions}).

\begin{thm}\label{thm:convergence}
Let $\Gamma$ be an oriented smooth convex arc in the upper unit half-disc $D_+$ with left endpoint $o=(-d,0)$, $d\in(0,1]$, where its curvature vanishes, and right endpoint on $\partial D$, which is met orthogonally. Suppose that the curvature of $\Gamma$ increases monotonically with arclength from $o$. If $d<1$, then the Dirichlet--Neumann curve shortening flow starting from $\Gamma$ exists for all positive time $t$ and converges uniformly in the smooth topology as $t\to\infty$ to the minimizing arc joining $o$ to $\partial D$. If $d=1$, then the Dirichlet--Neumann curve shortening flow starting from $\Gamma$ converges uniformly to the point $o$ as $t\to \omega<\infty$ and, after performing a standard type-II blow-up, converges locally uniformly in the smooth topology to a half Grim Reaper.
\end{thm}

Though the curvature monotonicity hypothesis appears unnaturally restrictive in Theorem \ref{thm:convergence}, we note that some such additional condition is required to prevent the development of self-intersections at the Dirichlet endpoint (resulting in subsequent cusplike singularities). Moreover, as Theorem \ref{thm:convergence} demonstrates in case the Dirichlet endpoint lies on the boundary, collapsing singularities may form at the Dirichlet endpoint even when the flow remains embedded. It is not hard to see that this can also occur when the Dirichlet endpoint lies to the interior (as a limiting case of the flow forming a cusp singularity just after losing embeddedness, say).

\subsection*{Acknowledgements} This work originated as an undergraduate summer research project undertaken at the ANU by YL and GM under the supervision of ML. We are grateful to the MSI for funding the project. ML thanks Julie Clutterbuck for bringing the mixed boundary value problem to his attention.

\section{Preliminaries}

Fix a point $o=(-d,0)\in D$ in the unit disc $D\subset \R^2$, with $d\in(0,1]$. Denote by $C_\theta\subset D$ the circular arc which passes through $o$ and meets the boundary of $D$ orthogonally at $(\sin\theta,\cos\theta)$; that is,
\[
C_\theta:=\{(x,y)\in D:(x-\xi)^2+(y-\eta)^2=r^2\}\,,
\]
where, defining $a:=\frac{1}{2}(d^{-1}+d)$,
\[
(\xi,\eta):=(\cos\theta,\sin\theta)+r(-\sin\theta,\cos\theta)\;\; \text{
and}\;\;
r:=\frac{1+d^2+2d\cos\theta}{2d\sin\theta}=\frac{a+\cos\theta}{\sin\theta}\,.
\]

Consider also the circular arc $\check C_{\theta}\subset D$ which is symmetric about the $y$-axis and meets $\partial D$ orthogonally at $(\cos\theta,\sin\theta)$. That is,
\[
\check C_\theta:=\{x^2+(y-\check\eta)^2=\check r^2\}\,,
\]
where
\[
\check\eta:=\csc\theta\;\;\text{and}\;\;\check r:=\cot\theta\,.
\]

\begin{prop}\label{prop:barriers}
The family $\{\check C_{\theta^+(t)}\}_{t\in(-\infty,0)}$, where $\theta^+(t):=\arcsin\mathrm{e}^{2t}$, is a supersolution to curve shortening flow. The family $\{C_{\theta^-(t)}\}_{t\in(-\infty,\omega_d)}$, where $\theta^-$ is the solution to
\begin{equation}\label{eq:characteristic ODE}
\left\{\begin{aligned}\frac{d\theta}{dt}={}&\frac{\sin\theta}{a+\cos\theta}\\ \theta(0)={}&\tfrac{\pi}{2}\,,\end{aligned}\right.
\end{equation}
is a subsolution to curve shortening flow.
\end{prop}
\begin{rem}
%Note that, when $d\in(0,1)$, $a>1$, and hence $\frac{\sin\theta}{a+\cos\theta}$ is uniformly bounded. So the solution to \eqref{eq:characteristic ODE} certainly exists for all $t\in(-\infty,\infty)$. Since $\frac{\sin\theta}{a+\cos\theta}$ is positive for $\theta\in(0,\pi)$, we also find that it must converge to zero as $t\to\pm\infty$, and hence, in particular,
%\[
%\lim_{t\to-\infty}\theta^-(t)=0\,.
%\]
%In fact, we can obtain the solution to \eqref{eq:characteristic ODE} explicitly: 
Separating variables, the problem \eqref{eq:characteristic ODE} becomes
%\bann
%-t=\int_t^0dt={}&\int_{\theta}^{\frac{\pi}{2}}\frac{a+\cos\omega}{\sin\omega}d\omega\\
%={}&\int_{\omega=\theta}^{\frac{\pi}{2}}\left(a\,d\log(\tan\tfrac{\omega}{2})+d\log\sin\omega\right)\\
%={}&\int_{\omega=\theta}^{\frac{\pi}{2}}d\log\left(\sin\omega\tan^a\tfrac{\omega}{2}\right)\\
%={}&\int_{\omega=\theta}^{\frac{\pi}{2}}d\log\left(2\sin^{1+a}\left(\tfrac{\omega}{2}\right)\cos^{1-a}\left(\tfrac{\omega}{2}\right)\right)
%\eann
\[
%-t=
\int_t^0dt=\int_{\theta}^{\frac{\pi}{2}}\frac{a+\cos\omega}{\sin\omega}d\omega
%=\int_{\omega=\theta}^{\frac{\pi}{2}}\left(a\,d\log(\tan\tfrac{\omega}{2})+d\log\sin\omega\right)
%={}&\int_{\omega=\theta}^{\frac{\pi}{2}}d\log\left(\sin\omega\tan^a\tfrac{\omega}{2}\right)\\
=\int_{\omega=\theta}^{\frac{\pi}{2}}d\log\left(2\sin^{1+a}\left(\tfrac{\omega}{2}\right)\cos^{1-a}\left(\tfrac{\omega}{2}\right)\right)\,,
\]
and hence
\[
\mathrm{e}^{t}=2\sin^{1+a}\left(\tfrac{\theta^-(t)}{2}\right)\cos^{1-a}\left(\tfrac{\theta^-(t)}{2}\right)\,.
\]
%Since, for $a>0$, the function $\Theta(\omega):=2\sin^{1+a}\left(\tfrac{\omega}{2}\right)\cos^{1-a}\left(\tfrac{\omega}{2}\right)$ is strictly monotone increasing for $\omega\in[0,\pi]$, it has a well-defined inverse, $\Theta^{-1}$, and hence
%\[
%\theta^-(t)=\Theta^{-1}(\mathrm{e}^t)\,.
%\]
In particular, for all $d\in(0,1]$, the solution certainly exists for all $t<0$, with $\theta^-(t)\sim 2^{\frac{a}{1+a}}\mathrm{e}^{\frac{t}{a+1}}$ as $t\to-\infty$. When $d\in(0,1)$, the solution exists up to time $\omega_d=+\infty$, and $\lim_{t\to+\infty}\theta^-(t)=\pi$. When $d=1$, the solution exists up to time $\omega_d=\log 2$, and $\lim_{t\to\omega_d}\theta^-(t)=\pi$.
%\[
%\lim_{t\to-\infty}\theta^-(t)=0\,.
%\]
%When $d=1$, the solution (for any $\theta^-(0)=\theta^-_0\in(0,\pi)$) is found to be
%\[
%\theta^-(t)=2\arcsin(A_0\mathrm{e}^{\frac{t}{2}})\,,\;\; A_0=\sin(\tfrac{\theta^-(0)}{2})\,,
%\]
%which certainly exists for $t\to-\infty$ with $\theta^-(t)\to 0$, but becomes singular as $t\nearrow \log A_0^{-2}$ (with $\theta^-(t)\to\pi$).
\end{rem}
\begin{proof}[Proof of Proposition \ref{prop:barriers}]
The first claim is proved in \cite[Proposition 2.1]{MR4668092}.

To prove the second claim, consider any monotone increasing function $\theta$ of $t$, and let $\gamma(u,t)=(x(u,t),y(u,t))$ be a general parametrization of $C_{\theta(t)}$. Differentiation of the equation
\[
(x-\xi)^2+(y-\eta)^2=r^2
\]
with respect to $t$ along $\gamma$ and $\theta$ yields
\[
(x-\xi)(x_t-\xi_\theta\theta_t)+(y-\eta)(y_t-\eta_\theta\theta_t)=rr_\theta\theta_t\,.
\]
Since the outward unit normal to $C_\theta$ at $(x,y)$ is $\nu=\frac{1}{r}(x-\xi,y-\eta)$, this becomes
\[
-\gamma_t\cdot\nu=-\left(\frac{x-\xi}{r}\xi_\theta+\frac{y-\eta}{r}\eta_\theta+rr_\theta\right)\theta_t\,.
\]
We claim that
\[
\frac{1}{r}(x-\xi,y-\eta)\cdot(\xi_\theta,\eta_\theta)+r_\theta=-\frac{y}{\sin\theta}\,.
\]
Indeed,
\begin{align*}
\frac{1}{r}({}&x-\xi,y-\eta)\cdot(\xi_\theta,\eta_\theta)\\
={}&\frac{1}{r}\Big((x,y)-(\cos\theta,\sin\theta)-r(-\sin\theta,\cos\theta))\Big)\cdot\Big((1+r_\theta)(-\sin\theta,\cos\theta)-r(\cos\theta,\sin\theta)\Big)\\
={}&-\left((x,y)-(\cos\theta,\sin\theta)-r(-\sin\theta,\cos\theta))\right)\cdot\left(\cot\theta(-\sin\theta,\cos\theta)+(\cos\theta,\sin\theta)\right)\\
={}&-(x,y)\cdot\left(\cot\theta(-\sin\theta,\cos\theta)+(\cos\theta,\sin\theta)\right)+1+r\cot\theta\\
={}&-(x,y)\cdot(0,\csc\theta)-r_\theta\,,
\end{align*}
from which the claim follows. %{\color{cyan}[I suspect that there is a nicer geometric proof of this.]}

Since, $y\le \sin\theta$ along $C_\theta$% (with equality at the Neumann endpoint)
, taking $\theta$ to be the solution to the specified initial value problem yields
\[
-\gamma_t\cdot\nu=\frac{y}{\sin\theta}\theta_t=\frac{y}{\sin\theta}\frac{1}{r}\le \frac{1}{r}=\kappa\,,
\]
as claimed.
\end{proof}

Next consider $\{\mathrm{H}_t\}_{t\in(-\infty,\infty)}$, the fundamental domain of the horizontally oriented hairclip solution to curve shortening flow centred at $o$; that is,
\[
\mathrm{H}_t:=\{(x,y)\in [0,\infty)\times[0,\tfrac{\pi}{2}]:\sin(y)=\mathrm{e}^{t}\sinh(x+d)\}\,.
\]
Given any $\lambda>0$, define $\{\mathrm{H}^\lambda_t\}_{t\in(-\infty,\infty)}$ by parabolically rescaling the hairclip by $\lambda$. That is,
\[
\mathrm{H}^\lambda_t:=\lambda^{-1}\mathrm{H}_{\lambda^{2}t}=\{(x,y)\in [0,\infty)\times[0,\tfrac{\pi}{2\lambda}]:\sin(\lambda y)=\mathrm{e}^{\lambda^2t}\sinh(\lambda(x+d))\}\,.
\]
Observe that $\{\mathrm{H}^\lambda_t\}_{t\in(-\infty,\infty)}$ satisfies
\[
\frac{\kappa}{\cos\theta}=\lambda\tan(\lambda y)\;\;\text{and}\;\;\frac{\kappa}{\sin\theta}=\lambda\tanh(\lambda(x+d))\,,
\]
where $\theta\in[0,\frac{\pi}{2}]$ is the angle the tangent vector makes with the $x$-axis. From this we see, in particular, that $\kappa$ is positive and monotone increasing with respect to arclength from $o$.
%{\color{red}could we instead consider $H_t^\lambda:=\lambda^{-1}H_{\lambda^{-2}t}$}

\begin{prop}
For each $\theta\in(0,\frac{\pi}{2})$ there exists a unique pair $(\lambda,t)$ such that $\mathrm{H}^\lambda_t$ intersects $\partial D$ orthogonally at $(\cos\theta,\sin\theta)$.
\end{prop}
\begin{proof}
Given any $\theta\in(0,\frac{\pi}{2})$, substituting the point $(\cos\theta,\sin\theta)$ for $(x,y)$ in the defining equation $\sin(\lambda y)=\mathrm{e}^{\lambda^2t}\sinh(\lambda(x+d))$ and solving for $t$ yields for each $\lambda\in(0,\frac{\pi}{2\sin\theta})$ the unique timeslice of the (fundamental domain of the) Hairclip solution which intersects $\partial D$ at $(\cos\theta,\sin\theta)$; namely,
\[
t=\lambda^{-2}\ln\left(\frac{\sin(\lambda\sin\theta)}{\sinh(\lambda(\cos\theta+d))}\right)\,.
\]
At that point, the normal satisfies
\begin{align*}
\nu_\lambda(\cos\theta,\sin\theta)\cdot(\cos\theta,\sin\theta)={}&\frac{\sin(\lambda\sin\theta)\cos\theta-\tanh(\lambda(\cos\theta+d))\cos(\lambda\sin\theta)\sin\theta}{\tanh(\lambda(\cos\theta+d))\cos(\lambda\sin\theta)}\\
={}&-\frac{\tan(\lambda\sin\theta)\cos\theta}{\tanh(\lambda(\cos\theta+d))}g(\lambda,\theta)\,,
\end{align*}
where
\[
g(\lambda,\theta):=\tanh(\lambda(\cos\theta+d))\cot(\lambda \sin\theta)\tan\theta-1\,.
\]
%For $(\lambda,\theta) \ \in \ (0,\frac{\pi}{2\sin\theta})\times(0,\frac{\pi}{2})$, $f(\lambda,\theta)=0$ if and only if $g(\lambda,\theta)=0$. 
Observe that
\[
\lim_{\lambda\searrow 0}g(\lambda,\theta)=d\cdot \sec\theta>0, \ \lim_{\lambda\nearrow\frac{\pi}{2\sin\theta}}g(\lambda,\theta)=-1<0
\]
and
\begin{align*}
\frac{dg}{d\lambda}={}&\tan\theta\Big[(\cos\theta+d)\cot(\lambda \sin\theta)\sech^2(\lambda(\cos\theta+d))\\
{}&\pushright{-\sin\theta \csc^2(\lambda \sin\theta)\tanh(\lambda(\cos\theta+d))\Big]}\\
={}&\frac{\tan\theta\tanh(\lambda(\cos\theta+d))}{\lambda\sin(\lambda\sin\theta)}\left[\frac{\lambda(\cos\theta+d)}{\sinh(\lambda(\cos\theta+d))}\cos(\lambda \sin\theta)\sech(\lambda(\cos\theta+d))\right.\\
{}&\pushright{\left.-\frac{\lambda\sin\theta}{\sin(\lambda \sin\theta)}\right]}\\
\le{}&\frac{\tan\theta\tanh(\lambda(\cos\theta+d))}{\lambda\sin(\lambda\sin\theta)}\left[\cos(\lambda \sin\theta)\sech(\lambda(\cos\theta+d))-1\right]\\
<{}&0
\end{align*}
for $\lambda \ \in \ (0,\frac{\pi}{2\sin\theta})$. It follows that there exists a unique $\lambda$ such that 
\[
\nu_\lambda(\cos\theta,\sin\theta)\cdot(\cos\theta,\sin\theta)=0\,.
\]
%Hence there is a unique pair $(\lambda,t)$ such that $\mathrm{H}_t^\lambda$ intersects $\partial D$ orthogonally at the point $(\cos\theta,\sin\theta)$.
The claim follows.
\end{proof}

\begin{rem}
Note that, since $\lim_{\theta\to 0}g(\lambda,\theta)=\frac{\tanh(\lambda(d+1)))}{\lambda}-1$, the function $g(\lambda,\theta)$ is non-negative at $\theta=0$ so long as $\lambda\geq \lambda_0$, where $\lambda_0$ is the unique positive solution to the equation
\[
\lambda=\tanh(\lambda(d+1))\,.
\]
\end{rem}

\begin{prop}[\emph{A priori} estimates]
Let $\Gamma\subset D_+$ be a smooth, convex embedding of a closed interval with left endpoint $o=(-d,0)$, $d\in(0,1]$, and right endpoint meeting $\partial D$ orthogonally, and suppose that the curvature $\kappa$ of $\Gamma$ increases monotonically with respect to arclength from $o$. Denote by $\underline\theta$ resp. $\overline\theta$ the least resp. greatest value taken by the turning angle along $\Gamma$ and by $\overline\kappa=\kappa(\overline\theta)$ the greatest value taken by $\kappa$.

The circle $C_{\overline\theta}$ lies below $\Gamma$. Thus,
\begin{equation}\label{eq:curvature lower bound}
\overline\kappa\ge\frac{\sin\overline\theta}{a+\cos\overline\theta}
\end{equation}
and
\begin{equation}\label{eq:gradient lower bound}
\underline\theta\ge\mathrm{arccot}\left(\frac{1+a\cos\overline\theta}{b\sin\overline\theta}\right)\,,
\end{equation}
where $b:=\frac{1}{2}(d^{-1}-d)$ and we recall that $a:=\frac{1}{2}(d^{-1}+d)$, with the right hand side taken to be zero in case $d=1$.
\end{prop}
\begin{proof}
Suppose, to the contrary, that $C_{\overline\theta}$ does not lie below $\Gamma$. Then some point of $\Gamma$ must lie strictly below $C_{\overline\theta}$, and hence (since the endpoints of the two curves agree) upon translating $C_{\overline\theta}$ downwards, the two curves will continue to intersect until some final moment, at which they must make first order contact at some interior point $q\in\Gamma$. At this point, the curvature $\kappa$ of $\Gamma$ must be no less than $1/r(\overline\theta)$ (the curvature of $C_{\overline\theta}$). But then, by the monotonicity of $\kappa$, $\kappa$ must exceed $1/r(\overline\theta)$ on the whole segment of $\Gamma$ joining $q$ to $\partial D$, in which case (since $\Gamma$ and $C_{\overline\theta}$ make first order contact at $\partial D$) the point $q$ must lie \emph{strictly above} $C_{\overline\theta}$, which is absurd. So $C_{\overline\theta}$ must indeed lie below $\Gamma$. The first inequality is then immediate and the second is straightforward.
\end{proof}

\begin{figure}[ht]
\includegraphics[width=0.6\textwidth]{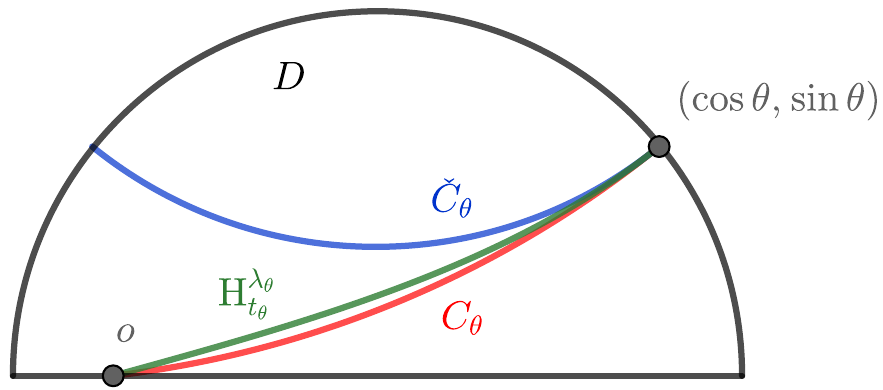}
\centering
\caption{Scaled hairclip timeslice and circular arcs through the prescribed boundary points $o$ and $(\cos\theta,\sin\theta)$.}
\end{figure}

\section{Existence}

For each $d\in(0,1]$ and $\rho\in(0,\frac{\pi}{2})$, let $\Gamma^\rho\subset D_+$ be a smooth oriented arc satisfying the following properties.
\begin{itemize}
\item The left endpoint of $\Gamma^\rho$ is $o=(-d,0)$, where its curvature vanishes, and its right endpoint meets $\partial D$ orthogonally at $(\cos\rho,\sin\rho)$.
\item $\Gamma^\rho$ is convex.
\item The curvature of $\Gamma^\rho$ is monotone increasing with respect to arclength from $o$.
\end{itemize}

For example, we could take $\Gamma^\rho:=\mathrm{H}^{\lambda_\rho}_{t_\rho}\cap D$, where $(\lambda_\rho,t_\rho)$ are the unique choice of $(\lambda,t)$ which ensure that $\mathrm{H}_t^\lambda$ meets $\partial D$ orthogonally at $(\cos\rho,\sin\rho)$.

\begin{lem}[Very old (but not ancient) solutions]\label{lem:old solutions}
For each $d\in(0,1]$ and $\rho\in(0,\frac{\pi}{2})$ there exist $\alpha_\rho<0$ such that $\alpha_\rho\to-\infty$ as $\rho\to 0$ and a smooth\footnote{More precisely,$\{\Gamma^\rho_t\}_{t\in[\alpha_\rho,\omega_d)}$ is given by a family of immersions of the interval $[0,1]$ which is of class $C^\infty([0,1]\times(\alpha_\rho,\omega_d))\cap C^{3+\beta,1+\frac{\beta}{2}}([0,1)\times[\alpha_\rho,\omega_d))\cap C^{2+\beta,1+\frac{\beta}{2}}((0,1]\times[\alpha_\rho,\omega_d))$ for any $\beta\in(0,1)$. Without additional compatibility conditions at the boundary points, higher regularity at the initial time may fail. However, if the curvature of $\Gamma^\rho$ is odd resp. even at its left resp. right boundary point, then the solution will be smooth up to the left resp. right boundary point at the initial time.} curve shortening flow $\{\Gamma^\rho_t\}_{t\in[\alpha_\rho,\omega_d)}$ in $D$ exhibiting the following properties.
\begin{itemize}
\item $\Gamma^\rho_{\alpha_\rho}=\Gamma^\rho$.
\item For each $t\in[\alpha_\rho,\omega_d)$, $\Gamma^\rho_t$ is an oriented embedding of a closed interval, with left endpoint $o=(-d,0)$ and right endpoint meeting $\partial D$ orthogonally.
\item For each $t\in[\alpha_\rho,\omega_d)$, $\Gamma^\rho_t$ is convex.
\item For each $t\in[\alpha_\rho,\omega_d)$, the curvature of $\Gamma^\rho_t$ is monotone increasing with respect to arclength from $o$.
\item If $d<1$, then $\omega_d=\infty$ and $\Gamma^\rho_t$ converges uniformly in the smooth topology as $t\to\infty$ to the minimizing arc $\{(x,0):x\in[-1,-d]\}$.
\item If $d=1$, then $\omega_d\in(0,\infty)$ and $\Gamma^\rho_t$ converges uniformly as $t\to\omega_d$ to the point $o$, and there are sequences of times $t_j\nearrow\omega_d$, points $p_j\in\Gamma^\rho_{t_j}$, and scales $\lambda_j\nearrow\infty$ such that the sequence $\{\lambda_j(\Gamma^\rho_{\lambda_j^{-2}t+t_j}-p_j)\}_{t\in[\lambda_j^2(\alpha_\rho-t_j),\lambda_j^2(\omega_\rho-j^{-1}-t_j))}$ converges locally uniformly in the smooth topology as $j\to\infty$ to the right half of the downwards translating Grim Reaper.
\end{itemize}
\end{lem}
\begin{proof}
%[Upon closer inspection of Stahl's paper \cite{Stahl96a}, it seems that he actually proves short time existence for very general barrier curves (they need not be boundaries of a convex domain).]
%
Form the ``odd doubling'' $\check\Gamma^\rho$ of $\Gamma^\rho$ by taking the union of $\Gamma^\rho$ with its rotation through angle $\pi$ about $o$. Since $\check\Gamma^\rho$ is a regular curve of class $C^2$ and there exists a ball $B$ about $(\cos\rho,\sin\rho)$ (of radius $1/10$, say) which is disjoint from the rotation of $\Gamma^\rho$ through angle $\pi$ about $o$, where $\Gamma^\rho$ meets $\partial D$ orthogonally, Stahl's short-time existence theorem for free boundary mean curvature flow \cite[Theorem 2.1]{Stahl96a} yields a solution $\{\check\Gamma^\rho_t\}_{t\in[0,\delta)}$ to Neumann--Neumann curve shortening flow with boundary on the odd doubling of $\partial D\cap B$ for a short time $\delta>0$. Since this solution is uniquely determined by its initial condition, it must be invariant under rotation through angle $\pi$ about $o$, and hence descend to a solution $\{\hat\Gamma^\rho_t\}_{t\in[0,\delta)}$ to Dirichlet--Neumann curve shortening flow in $D$ with Dirichlet condition $o$ and initial condition $\Gamma^\rho$. Denote by $T$ the maximal time of existence of the latter.

Since the curvature of $\{\hat\Gamma^\rho_t\}_{t\in[0,T)}$ satisfies
\[
\left\{\begin{aligned}(\partial_t-\Delta)\kappa={}&\kappa^3\\
\kappa={}&0\;\;\text{at $o$, and}\\
\kappa_s={}&\kappa\;\;\text{at}\;\;\partial D\,,
\end{aligned}\right.
\]
where $s$ denotes arclength from $o$, the maximum principle (and Hopf boundary point lemma) ensure that $\kappa$ remains positive on $\hat\Gamma^\rho_t\setminus\{o\}$ for $t>0$.

For similar reasons, positivity of $\kappa_s$ is also preserved. Indeed, using the commutator relation
\[
[\partial_t,\partial_s]=\kappa^2\partial_s\,,
\]
the identity $0=\kappa_t=\Delta\kappa$ at $o$, and the positivity of $\kappa$ away from $o$, we find that
\[
\left\{\begin{aligned}(\partial_t-\Delta)\kappa_s={}&4\kappa^2\kappa_s\\
(\kappa_s)_s={}&0\;\;\text{at $o$, and}\\
\kappa_s>{}&0\;\;\text{at}\;\;\partial D\,,
\end{aligned}\right.
\]
so the claim once again follows from the maximum principle.

Since $\overline\theta_t=\overline\kappa>0$ and $\overline\theta<\pi$ (when $d<1$, the maximum principle prevents $\hat\Gamma^\rho_t$ from ever reaching the minimizing arc --- a stationary solution to the flow) we find that $\overline\theta$ must attain a limit as $t\to T$. We claim that this limit is $\pi$. Indeed, if $\overline\theta\le \theta_0<\pi$ for all $t\in[0,T)$, then, representing the solution as a graph over the line $\{(-d,y):y\in\R\}$, the ``gradient estimate'' \eqref{eq:gradient lower bound} yields a uniform bound for the gradient, at least when $d<1$. But then, by applying parabolic regularity theory (see, for instance, \cite{Lieberman}) to the graphical Dirichlet--Neumann curve shortening flow equation
\[
\left\{\begin{aligned}x_t={}&\frac{x_{yy}}{1+x_y^2}\;\;\text{in}\;\; [0,\overline y(t)]\\
x(0,t)={}&0\\
x_y(\overline y(t),t)={}&\cot\overline\theta(t)\,,
\end{aligned}\right.
\]
where $\overline y(t):=\sin\overline\theta(t)$, we obtain uniform estimates for all derivatives of the graph functions $x(\cdot,t)$ (cf. \cite{Stahl96a}). To obtain corresponding estimates when $d=1$, we instead represent the solution as a graph over the ``tilted'' line through $(-1,0)$ and $(\cos(\overline\theta(T)),\sin(\overline\theta(T)))$ and use the ``gradient estimate'' $\underline\theta\ge 0$. The Arzel\`a--Ascoli theorem and monotonicity of the flow now ensure that $x(\cdot,t)$ takes a smooth limit as $t\to T$, at which point the flow can be smoothly continued by the above short time existence argument, violating the maximality of $T$. We conclude that $\overline\theta(t)\to \pi$ as $t\to T$. 

It now follows from \eqref{eq:gradient lower bound} that $\underline\theta(t)\to \pi$ as $t\to T$. When $d=1$, we conclude that $\hat\Gamma^\rho_t$ contracts to $o$ as $t\to T$. Note that in this case $T<\infty$ since the lower barriers $C_{\theta^-(t)}$ contract to $o$ in finite time. A more or less standard ``type-I vs type-II" blow-up argument (cf. \cite{FBDistanceComparison}) then guarantees convergence to the half Grim Reaper after performing a standard type-II blow-up. (The flow must be type-II because the limit of a standard type-I blow-up --- a shrinking semi-circle --- violates the Dirichlet boundary condition.% The type-II blow-up then yields a half Grim Reaper, rather than a Grim Reaper, as the curvature is maximized at the Neumann boundary.
)

When $d<1$, we conclude that $\hat\Gamma^\rho_t$ converges uniformly to the minimizing arc $\{(x,0):x\in[-1,-d]\}$ as $t\to T$. But then, for large enough $t$, $\hat\Gamma^\rho_t$ may be represented as a graph over the $x$-axis with small gradient, at which point parabolic regularity, short-time existence and the Arzel\`a--Ascoli theorem guarantee that $T=\infty$ and $\hat\Gamma^\rho_t$ converges uniformly in the smooth topology to the minimizing arc.

Finally, since $\overline\theta$ is monotone, there is a unique time $-\alpha_\rho>0$ such that $\overline\theta(-\alpha_\rho)=\frac{\pi}{2}$; since the Neumann--Neumann circle $\check C_{\theta_\rho}$, where $\sin\theta_\rho=\frac{2\sin\rho}{1+\sin^2\rho}$, lies above $\Gamma^\rho$, we find (by suitably time translating the upper barrier $\{\check C_{\theta^+(t)}\}_{t\in(-\infty,0)}$, as in \cite{MR4668092}) that
\[
\alpha_\rho<\frac{1}{2}\log\left(\frac{2\sin\rho}{1+\sin^2\rho}\right)\,.
\]
Time-translating the solution $\{\hat\Gamma^\rho\}_{t\in[0,\infty)}$ by $\alpha_\rho$ now yields the desired very old (but not ancient) solution $\{\Gamma^\rho_t\}_{t\in[\alpha_\rho,\infty)}$.
\end{proof}

%Given $t\in[\alpha,\infty)$, let $\underline\theta(t)$ resp. $\overline\theta(t)$ be the least resp. greatest value of the turning angle along $\Gamma_t$. %Note that the Neumann condition ensures that $(\cos\overline\theta(t),\sin\overline\theta(t))$ is the Neumann endpoint of $\Gamma_t$.

Taking the limit as $\rho\to 0$ of these very old (but not ancient) solutions yields our desired ancient solution.

\begin{thm}\label{thm:existence}
Given any $d\in(0,1]$, there exists a convex ancient Dirichlet--Neumann curve shortening flow $\{\Gamma_t\}_{t\in(-\infty,\omega_d)}$ in the upper half disc $D_+$ which converges uniformly in the smooth topology to the unstable critical arc $[-d,1]\times \{0\}$ as $t\to-\infty$. When $d<1$, $\omega_d=\infty$ and $\{\Gamma_t\}_{t\in(-\infty,\omega_d)}$ converges uniformly in the smooth topology as $t\to+\infty$ to the minimizing arc $[-1,-d]\times \{0\}$. When $d=1$, $\omega_d<\infty$ and $\Gamma^\rho_t$ converges uniformly as $t\to\omega_d$ to the point $o$, and there are a sequence of times $t_j\nearrow\omega_d$, right endpoints $p_j\in\Gamma^\rho_{t_j}$, and scales $\lambda_j\nearrow\infty$ such that the sequence $\{\lambda_j(\Gamma^\rho_{\lambda_j^{-2}t+t_j}-p_j)\}_{t\in[\lambda_j^2(\alpha_\rho-t_j),\lambda_j^2(\omega_\rho-j^{-1}-t_j))}$ converges locally uniformly in the smooth topology as $j\to\infty$ to the right half of the downwards translating Grim Reaper.
\end{thm}
\begin{proof}
Given any sequence of angles $\rho_j\searrow 0$, consider the sequence of corresponding very old (but not ancient) solutions $\{\Gamma^j_t\}_{t\in[\alpha_j,\infty)}$ constructed in Lemma \ref{lem:old solutions}. Differentiating the Neumann boundary condition and applying the estimate \eqref{eq:curvature lower bound} yields the inequality
\[
\overline\theta_t=\overline\kappa\ge \frac{\sin\overline\theta}{a+\cos\overline\theta}
\]
on each of these solutions. It follows, by the ODE comparison principle, that each $\{\Gamma^j_t\}_{t\in[\alpha_j,\infty)}$ satisfies
\begin{equation}\label{eq:gradient upper bound}
\overline\theta\le \theta^-
\end{equation}
for $t\in[\alpha_j,0]$, where we recall that $\theta^-$ is the solution to \eqref{eq:characteristic ODE}. Since $\theta^-$ is independent of $j$, this implies uniform estimates for the gradient on any time interval of the form $[-\infty,-T]$, $T>0$, when we represent $\{\Gamma^j_t\}_{t\in[\alpha_j,-T]}$ graphically over the $x$-axis. By parabolic regularity theory and the Arzel\`a--Ascoli theorem, we may then extract a smooth limit of the very old solutions $\{\Gamma^{j}_t\}_{t\in(-\infty,0)}$ after passing to a subsequence. This limit is ancient, since $\alpha_\rho\to-\infty$ as $\rho\to 0$, reaches the point $(0,1)$ at time zero (since each $\Gamma^{j}_t$ intersects the convex domain bounded by $\check C_{\theta^+(t)}$ for each $t\in(-\infty,0)$), and converges uniformly in the smooth topology as $t\to-\infty$ to the unstable Dirichlet--Neumann geodesic from $o$ due to the estimate \eqref{eq:gradient upper bound} (and parabolic regularity theory). The longtime behaviour %continues to exist for $t\to+\infty$, converging to the stable Dirichlet--Neumann geodesic from $o$, by the long time existence 
follows from the argument presented above.
\end{proof}

\subsection{Asymptotics}

We now prove precise asymptotics for the height of the ancient solution constructed in Theorem \ref{thm:existence}, assuming the initial conditions for the old-but-not-ancient solutions $\{\Gamma^\rho_t\}_{t\in[\alpha_\rho,\omega_d)}$ are given by the hairclip timeslices $\Gamma^\rho=\mathrm{H}^{\lambda_\rho}_{t_\rho}\cap D$.
\begin{lem}\label{lem:sharp speed lower bound}
On each old-but-not-ancient solution $\{\Gamma^\rho_t\}_{t\in[\alpha_\rho,\omega_d)}$,
\begin{equation*}%\label{eq:sharp speed lower bound}
\frac{\kappa}{\cos\theta}\ge \lambda_\rho\tan(\lambda_\rho y)\,.
\end{equation*}
\end{lem}
%{\color{cyan}
%[I think a similar argument gives $\frac{\kappa}{\sin\theta}\le\lambda\tanh(\lambda(x+d))$, at least for a suitable range of $t$. But I'm not sure how to use this --- the cruder estimates below are sufficient for our purpose.]}
\begin{proof}%[Proof of Lemma \ref{lem:sharp speed lower bound}]
Note that equality holds on the initial curve $\Gamma^\rho=\mathrm{H}^{\lambda_\rho}_{t_\rho}\cap D$. Thus, given any $\mu<\lambda_\rho$, the function
\[
w:=\frac{\kappa}{\cos\theta}-\mu\tan(\mu y)
\]
is strictly positive on the initial curve $\Gamma^\rho$, except at the left endpoint, where it vanishes. Observe that
\[
w_s= \frac{\kappa_s}{\cos\theta}+\sin\theta\left(\frac{\kappa^2}{\cos^2\theta}-\mu^2\sec^2(\mu y)\right)\,.
\]
In particular, at the left endpoint on the initial curve,
\begin{align*}
w_s={}&\frac{\kappa_s}{\cos\theta}-\mu^2\sin\theta=(\lambda_\rho^2-\mu^2)\sin\theta>0\,.
\end{align*}
Thus (since $w_s$ is continuous at $o$ at time zero), if $w$ fails to remain non-negative at positive times, then this failure must occur immediately following some \emph{interior} time $t_\ast>0$. There are three possibilities: 1. $w_s(\cdot,t_\ast)=0$ at the left endpoint, 2. $w(\cdot,t_\ast)=0$ at the right endpoint; or 3. $w(\cdot,t_\ast)=0$ at some interior point, $p_\ast$. 

The first of the three possibilities is immediately ruled out by the Hopf boundary point lemma.

In the second case, the Hopf boundary point lemma and the Neumann boundary condition yield, at the right endpoint,
\begin{align*}
0>w_s%={}&\frac{\kappa}{\cos\theta}+\frac{\kappa^2\sin\theta}{\cos^2\theta}-\mu^2\sec^2(\mu y)\sin\theta\\
={}&\frac{\kappa}{\cos\theta}+\sin\theta\left(\frac{\kappa^2}{\cos^2\theta}-\mu^2(1+\tan^2(\mu y))\right)=\mu\tan(\mu y)-\mu^2y\ge0\,,
\end{align*}
which is absurd.

In the final case (having ruled out the first two), $w$ must attain a negative interior minumum just following time $t_\ast$. But at such a point, $w<0$, $w_s=0$ and
\begin{align*}
0\ge{}& (\partial_t-\Delta)w\\
={}&\frac{(\partial_t-\Delta)\kappa}{\cos\theta}-\frac{\kappa(\partial_t-\Delta)\cos\theta}{\cos^2\theta}+2\left(\frac{\kappa}{\cos\theta}\right)_s\frac{(\cos\theta)_s}{\cos\theta}-(\partial_t-\Delta)(\mu\tan(\mu y))\,.
\end{align*}
Since
\[
(\partial_t-\Delta)\kappa=\kappa^3\,,\;\;(\partial_t-\Delta)\cos\theta=\kappa^2\cos\theta\;\;\text{and}\;\;(\partial_t-\Delta)y=0\,,
\]
we conclude that
\begin{align*}
0\ge{}&-2\left(\frac{\kappa}{\cos\theta}\right)_s\kappa\tan\theta+2\mu\tan(\mu y)(\mu\tan(\mu y))_s\sin\theta\\
={}&2(\mu\tan(\mu y))_s\sin\theta\left(\mu\tan(\mu y)-\frac{\kappa}{\cos\theta}\right)\\
={}&2\mu^2\sec^2(\mu y)\left(\mu\tan(\mu y)-\frac{\kappa}{\cos\theta}\right)\\
>{}&0\,,
\end{align*}
which is absurd. 

Having ruled out each of the three possibilities, we conclude that $w\ge 0$ for any $\mu<\lambda_\rho$. The claim follows.
\end{proof}

In the limit as $\rho\to 0$, we then obtain
\begin{equation}\label{eq:sharp speed lower bound}
\frac{\kappa}{\cos\theta}\ge \lambda_0\tan(\lambda_0y)
\end{equation}
on the ancient solution, where $\lambda_0=\lim_{\rho\to 0}\lambda_\rho$. %(which coincides with the ground frequency of the Dirichlet--Robin Laplacian on the interval $[-d,1]$).

We now find that, as a graph over the $x$-axis (for $t$ sufficiently negative),
\[
(\sin(\lambda_0y))_t=\lambda\cos(\lambda_0 y)y_t=\lambda_0\cos(\lambda_0 y)\sqrt{1+y_x^2}\kappa=\lambda_0\cos(\lambda_0 y)\frac{\kappa}{\cos\theta}\ge\lambda_0^2\sin(\lambda y)\,,
\]
and hence
\[
\left(\mathrm{e}^{-\lambda_0^2t}\sin(\lambda_0 y)\right)_t\ge 0\,,
\]
which implies that the limit
\[
A(x):=\lim_{t\to-\infty}\mathrm{e}^{-\lambda_0^2t}y(x,t)
\]
exists in $[0,\infty)$ for each $x\in(-d,1)$. 

%Observe now that, for $t\sim -\infty$, $\frac{\sin\theta}{a+\cos\theta}\sim \frac{\theta}{a+1}$, and hence 
Recall that $\theta^-(t)\sim \mathrm{e}^{\frac{t}{a+1}}$ for $t\sim -\infty$. In particular, $\overline\theta(t)\le\theta^-(t)$ %(and hence also $\sin\overline\theta(t)\le\sin\theta^-(t)$) 
is integrable. We will exploit this fact to show that the limit $A(x)$ is positive (at least near $x=1$). First, we shall show that $\overline\kappa$ is integrable.

\begin{lem}\label{lem:curvature is integrable}
There exist $\rho_0>0$, $T>-\infty$, $C<\infty$ and $\delta>0$ such that
\[
\kappa\le C\mathrm{e}^{\delta t}\;\;\text{for}\;\; t\le T
\]
on each old-but-not-ancient solution $\{\Gamma^\rho_t\}_{t\in[\alpha_\rho,\omega_d)}$ with $\rho<\rho_0$.
\end{lem}
\begin{proof}
Since
\[
(\partial_t-\Delta)\sin\theta=\kappa^2\sin\theta\,,
\]
we find that
\[
(\partial_t-\Delta)\frac{\kappa}{\sin\theta}=2\nabla\frac{\kappa}{\sin\theta}\cdot\frac{\nabla\sin\theta}{\sin\theta}\,.
\]
So the maximum principle guarantees that the maximum of $\frac{\kappa}{\sin\theta}$ occurs at the parabolic boundary. Now, at the left boundary point, $\frac{\kappa}{\sin\theta}=0$, while at the right,
\begin{align*}
\left(\frac{\kappa}{\sin\theta}\right)_s={}&\frac{\kappa}{\sin\theta}\left(\frac{\kappa_s}{\kappa}-\frac{\cos\theta\kappa}{\sin\theta}\right)=\frac{\overline\kappa}{\sin\overline\theta}\left(1-\frac{\overline\kappa}{\tan\overline\theta}\right)\,.
\end{align*}
By \eqref{eq:gradient upper bound}, we can find $T>-\infty$ (independent of $\rho$) so that $\cos\overline\theta(t)\ge \frac{1}{2}$ for all $t\le T$. We thereby conclude that
\[
\frac{\kappa}{\sin\theta}\le\max\left\{2,\max_{t=\alpha_\rho}\frac{\kappa}{\sin\theta}\right\}
\]
for all $t\le T$. Since $\max_{t=\alpha_\rho}\frac{\kappa}{\sin\theta}\le \lambda_\rho\tanh(\lambda_\rho(1+d))\to\lambda_0^2<1$ as $\rho\to 0$, we find that
\begin{equation}\label{eq:C2 estimate constructed solution}
\overline\kappa\le 2\sin\overline\theta\le 2\sin\theta^-
\end{equation}
for all $\rho$ sufficiently small. The claim follows since $\theta^-$ is comparable to $2^{1+\frac{a}{1+a}}\mathrm{e}^{\frac{t}{1+a}}$ as $t\to-\infty$.
\end{proof}

\begin{cor}\label{cor:sharp speed upper bound}
There exist $T>-\infty$, $C<\infty$ and $\delta>0$ such that
\[
\frac{\kappa}{y}\le \lambda_0^2+C\mathrm{e}^{\delta t}\;\;\text{for}\;\; t<T
\]
%on each old-but-not-ancient solution $\{\Gamma^\rho_t\}_{t\in[\alpha_\rho,\infty)}$ with $\rho<\rho_0$.
%There exists $\overline a>0$ such that
%\[
%\overline y\ge \overline a\mathrm{e}^{\lambda_0^2 t}
%\]
on the ancient solution.
\end{cor}
\begin{proof}
Given $\rho<\rho_0$, set
\[
\eta_\rho(t):=\lambda_\rho^2\left(\exp\big(\tfrac{C^2}{2\delta}\mathrm{e}^{2\delta t}\big)-1\right)\,,
\]
where $\rho_0$, $C$ and $\delta$ are the constants from Lemma \ref{lem:curvature is integrable}, so that
\[
\frac{\eta_\rho'}{\lambda_\rho^2+\eta_\rho}=C^2\mathrm{e}^{2\delta t}
\]
and hence, for $t<T$,
\begin{align*}
(\partial_t-\Delta)\big(\kappa-(\lambda_\rho^2+\eta_\rho)y\big)={}&\kappa^3-\eta_\rho'y\\
\le{}&C^2\mathrm{e}^{2\delta t}\kappa-\frac{\eta_\rho'}{\lambda_\rho^2+\eta_\rho}(\lambda_\rho^2+\eta_\rho)y\\
={}&C^2\mathrm{e}^{2\delta t}\big(\kappa-(\lambda_\rho^2+\eta_\rho)y\big)\,.
\end{align*}
Since $\kappa-(\lambda_\rho^2+\eta_\rho)y=0$ at the left endpoint and $(\kappa-(\lambda_\rho^2+\eta_\rho)y)_s=\kappa-(\lambda_\rho^2+\eta_\rho)y$ at the right endpoint, we find that
\[
\kappa-(\lambda_\rho^2+\eta_\rho)y\le \exp\left(\frac{C^2}{2\delta}\mathrm{e}^{2\delta t}\right)\big(\kappa-(\lambda_\rho^2+\eta_\rho)y\big)\Big|_{t=\alpha_\rho}
\]
on each of the old-but-not-ancient solutions with $\rho<\rho_0$, and hence, taking $\rho\to0$,
\[
\kappa\le (\lambda_0^2+\eta_0)y
\]
on the ancient solution. The claim follows since, by the mean value theorem, we may estimate $\eta_0\le \frac{\lambda_0^2C^4}{4\delta^2}\mathrm{e}^{2\delta t}$ for $t<0$.
\end{proof}

By the estimate \eqref{eq:gradient upper bound} and Corollary \ref{cor:sharp speed upper bound}, we can find $T>-\infty$, $C<\infty$ and $\delta>0$ such that our ancient solution satisfies
\[
(\log \overline y)_t=\frac{\overline \kappa}{\overline y\cos\overline\theta}\le \frac{1}{\sqrt{1-4C^2\mathrm{e}^{2\delta t}}}\frac{\overline \kappa}{\overline y}\le (1+8C^2\mathrm{e}^{2\delta t})\frac{\overline \kappa}{\overline y}%\le (1+8C^2\mathrm{e}^{2\delta t})(\lambda_0^2+C\mathrm{e}^{\delta t})
\le \lambda_0^2+C^4\mathrm{e}^{2\delta t}
\]
for $t<T$. Integrating from time $t<T$ to time $T$ and rearranging then yields
\[
\overline y\ge B\mathrm{e}^{\lambda_0^2t}\,,\;\; B>0\,.
\]
Since the gradient of the solution is bounded by $\tan\overline\theta \le C\mathrm{e}^{\lambda_0^2t}$ for $t\le T$, this guarantees that the limit $A(x):=\mathrm{e}^{-\lambda_0^2t}y(x,t)$ is positive for all $x>x_0$ where $x_0<1$.

%This is not quite good enough. We should try, instead, to compare $\kappa$ to $\Theta$: note that
%\begin{align*}
%\left(\frac{\kappa}{\Theta(\theta)}\right)_s={}&\frac{\kappa}{\Theta(\theta)}\left(\frac{\kappa_s}{\kappa}-\frac{\Theta'(\theta)}{\Theta(\theta)}\theta_s\right)\\
%={}&\frac{\kappa}{\Theta(\theta)}\left(\frac{\kappa_s}{\kappa}-\frac{a+\cos\theta}{\sin\theta}\kappa\right)\\
%\le{}&0
%\end{align*}
%at the right boundary point. So consider
%\begin{align*}
%(\partial_t-\Delta)\Theta(\theta)={}&-\kappa^2\Theta''(\theta)\\
%={}&\kappa^2\left(1-\frac{a}{\sin\theta}\frac{a+\cos\theta}{\sin\theta}\right)\Theta(\theta)
%\end{align*}
%(where we used $\theta_t=\kappa_s$, $\theta_s=\kappa$, and $\Theta'=\frac{a+\cos\theta}{\sin\theta}\Theta$). Thus,
%\[
%(\partial_t-\Delta)\frac{\kappa}{\Theta(\theta)}=\frac{\kappa}{\Theta(\theta)}\kappa^2\frac{a}{\sin\theta}\frac{a+\cos\theta}{\sin\theta}+2\left(\frac{\kappa}{\Theta}\right)_s\frac{\Theta_s}{\Theta}
%\]

\section{Uniqueness}

\subsection{Unique asymptotics}

Consider now any convex ancient Dirichlet--Neumann curve shortening flow $\{\Gamma_t\}_{t\in(-\infty,\omega)}$ with Dirichlet endpoint $o\in \overline D\setminus\{0\}$.

\begin{lem}
Up to a time-translation, a rotation about the origin, and a reflection about the $x$-axis, we may arrange that
\begin{itemize}
\item $o=(-d,0)$ for some $d\in(0,1]$,
\item $(0,1)\in\Gamma_0$,
\item $\Gamma_t$ lies in the upper half disc for all $t$, and
\item $\Gamma_t\to \{(x,0):x\in[-d,1]\}$ uniformly in the smooth topology as $t\to-\infty$.
\end{itemize}
\end{lem}
\begin{proof}
Up to a time translation, we may arrange that $\omega>0$. Up to rotation and a reflection, we may then arrange that $o=(-d,0)$, $d\in(0,1]$, and $(\cos\overline\theta(0),\sin\overline\theta(0))$ lies in the upper half-disc. This ensures that $(\cos\overline\theta(t),\sin\overline\theta(t))$ lies in the upper half-disc for all $t<0$. Indeed, if $\overline\theta(t_\ast)=0$ for some $t_\ast<0$, then convexity and the boundary conditions guarantee that $\Gamma_{t_\ast}=\{(x,0):x\in [-d,1]\}$; so $\{\Gamma_t\}_{t\in(-\infty,\omega)}$ is the stationary unstable critical arc.

Denote by $\Omega_t$ the region lying above $\Gamma_t\cup\{(0,x):x\in[-1,-d]\}$ and set $\Omega:=\cup_{t<\omega}\Omega_t$. The first variation formula for enclosed area yields
\[
\frac{d}{dt}\area(\Omega_t)=-\int_{\Gamma_t}\kappa\,ds=-(\overline\theta(t)-\underline\theta(t))
\]
and hence
\[
\area(\Omega_t)=\area(\Omega_0)+\int_t^0(\overline\theta(\tau)-\underline\theta(\tau))\,d\tau\,.
\]
Since $\area(\Omega)$ is finite, $\overline\theta-\underline\theta$ must converge to zero along some sequence of times $t_j\to-\infty$. Since $\overline\theta>0$, this ensures that $\Omega$ is the upper half-disc, and hence $\Gamma_t$ converges uniformly to the unstable critical arc as $t\to-\infty$. Parabolic regularity theory then guarantees smooth convergence.

%Let $t_0$ be the time at which $\theta^-(t_0)=\overline\theta(0)$. Since the Dirichlet--Neumann circular arcs $C_{\theta^-(t+t_0)}$ form a subsolution to Dirichlet--Neumann curve shortening flow, the curve $C_{\theta^-(t+t_0)}$ must intersect the region $\Omega_t$ lying above $\Gamma_t\cup\{(0,x):x\in[-1,-d]\}$ for each $t<0$. It follows that $\cup_{t\in(-\infty,\infty)}\Omega_t$ is the upper half-disc. Since the flow is monotone, this means that $\Gamma_t$ converges uniformly to the unstable critical arc $\{(x,0):x\in[-d,1]\}$ as $t\to-\infty$. Smooth convergence then follows from the parabolic theory, as described above. 

Since the flow is monotone, $\Gamma_t$ must then lie in the upper half disc for all $t$. We have thus shown, when $d<1$, that $\omega=\infty$ and $\Gamma_t$ converges smoothly to the minimizing arc as $t\to\infty$ and, when $d=1$, that $\omega<\infty$ and $\Gamma_t$ converges uniformly to $o$ as $t\to \omega$. Up to a further time-translation, we may therefore arrange that the point $(0,1)$ lies in $\Gamma_0$. 
\end{proof}

\begin{lem}
For every $t\in(-\infty,\omega)$, $\kappa_s>0$.
\end{lem}
\begin{proof}
Since $\kappa_s\ge 0$ at both endpoints, the claim may be obtained by applying the maximum principle exactly as in \cite[Lemma 3.3]{MR4668092}.
\end{proof}

\begin{prop}
%Every convex ancient Dirichlet--Neumann curve shortening flow in $D_+$ with Dirichlet endpoint $o=(-d,0)$, $d\in(0,1]$, satisfies
There exists $A\in[0,\infty)$ such that
\begin{equation}\label{eq:asymptotics}
y(x,t)=A\mathrm{e}^{\lambda_0^2t}(\sinh(\lambda(x+d))+o(1))
\end{equation}
uniformly as $t\to-\infty$. %On the constructed solution, $A>0$.%for some $A\in[0,\infty)$ (with $A>0$ on the constructed solution).
%The constant $A$ is positive on at least one solution.
\end{prop}
\begin{proof}
Denote by $\{\Gamma^\ast_t\}_{t\in(-\infty,\omega)}$ the constructed solution. Since $\Gamma_0$ and $\Gamma_0^\ast$ both contain the point $(0,1)$, the contrapositive of the avoidance principle guarantees that $\Gamma_t$ must intersect $\Gamma^\ast_t$ away from $o$ at every time $t<0$. It follows that the value of $\underline\theta$ on the second solution must at no time exceed the value of $\overline\theta{}^\ast$ on the constructed solution. But then, applying the gradient bound \eqref{eq:gradient lower bound} and estimating $\sin\overline\theta{}^\ast\le A\mathrm{e}^{\lambda_0^2t}+o(\mathrm{e}^{\lambda_0^2t})$ yields
\[
\frac{b}{1+a}\overline y\le \frac{b\sin\overline\theta}{1+a\cos\overline\theta}\le\tan\underline\theta\le 2\sin\underline\theta\le 2\sin\overline\theta{}^\ast\le 2(A\mathrm{e}^{\lambda_0^2t}+o(\mathrm{e}^{\lambda_0^2t}))
\]
as $t\to-\infty$, and hence, when $d<1$,
\begin{equation}\label{eq:unique height asymptotics}
\limsup_{t\to-\infty}\mathrm{e}^{\lambda_0^2t}\overline y(t)<\infty.
\end{equation}
Since the height function $y$ satisfies the (intrinsic) Dirichlet--Robin heat equation
\[
\left\{\begin{aligned}(\partial_t-\Delta)y={}&0\\ y=0\;\;\text{at}\;\; o\,, {}&\; y_s=y\;\;\text{at}\;\;(\cos\overline\theta,\sin\overline\theta)\,,\end{aligned}\right.
\]
we may apply Alaoglu's theorem and elementary Fourier analysis as in \cite[Proposition 3.4]{MR4668092} to obtain \eqref{eq:asymptotics}.

When $d=1$, we need to work a little harder to obtain \eqref{eq:unique height asymptotics}: at any time $t<0$, either $\overline y(t)\le \overline y{}^\ast(t)$, as desired, or $\overline y(t)>\overline y{}^\ast(t)$. In the latter case, the avoidance principle and the Dirichlet condition ensure that $y{}^\ast(\cdot,t)-y(\cdot,t)$ attains a positive maximum at an interior point. Since the Dirichlet--Neumann circular arc $C_{\overline \theta(t)}$ lies below $\Gamma_t$ (with common boundary), we can find some $t_0>t$ and $x_0\in(-1,\cos\overline\theta(t_0))$ such that the advanced arc $C_{\overline \theta(t_0)}$ touches $\Gamma_t^\ast$ from above at the interior point $(x_0,y^\ast(x_0,t))$, and hence
\begin{equation*}%\label{eq:avoidance barrier argumet d=1}
y_{\overline\theta(t_0)}(x_0)=y^\ast(x_0,t)=:A_0\;\;\text{and}\;\; (y_{\overline\theta(t_0)})_x(x_0)= y^\ast_x(x_0,t)=:B_0\,,
\end{equation*}
where
\[
y_{\theta}(x)=r(\theta)-\sqrt{r^2(\theta)-(x+1)^2}\,,\;\; r(\theta):=\frac{1+\cos\theta}{\sin\theta}\,.
\]
That is,
\[
r_0-\sqrt{r^2_0-(x_0+1)^2}=A_0\;\;\text{and}\;\;\frac{x_0+1}{\sqrt{r^2_0-(x_0+1)^2}}=B_0\,,
\]
where $r_0:= r(\overline\theta(t_0))$. Rearranging, these become
\[
%r_0=\frac{A_0^2+(x_0+1)^2}{2A_0}
x_0+1=\frac{B_0r_0}{\sqrt{1+B_0^2}} \;\;\text{and}\;\; r_0=A_0+\frac{x_0+1}{B_0}=A_0+\frac{r_0}{\sqrt{1+B_0^2}}\,,
\]
which together imply that
\[
\left(\sqrt{1+B_0^2}-1\right)r_0=A_0\sqrt{1+B_0^2}=\frac{A_0}{x_0+1}B_0r_0\,.
\]
Eliminating $r_0$ and rearranging, we conclude that
\begin{equation}\label{eq:intersection points for unique asymptotics when d=1}
\frac{A_0}{(x_0+1)B_0}=\frac{1}{1+\sqrt{1+B_0^2}}\,.
\end{equation}
We claim that this is only possible (when $-t$ is sufficiently large) if $x_0$ is close to one. Indeed, the asymptotic linear analysis yields, for some $A\in(0,\infty)$,
\[
\left\{\begin{aligned}
A_0={}&y^\ast(x_0,t)=A\mathrm{e}^{\lambda_0^2t}\left(\sinh(\lambda_0(x_0+1))+o(1)\right)\\
B_0={}&y^\ast_x(x_0,t)=A\lambda_0\mathrm{e}^{\lambda_0^2t}\left(\cosh(\lambda_0(x_0+1))+o(1)\right)
\end{aligned}\right.\;\;\text{as}\;\; t\to-\infty\,.
\]
(Note that, recalling \eqref{eq:C2 estimate constructed solution}, we may estimate $y^\ast_{xx}\lesssim \overline\kappa{}^\ast\le 2\sin\overline\theta{}^\ast\le C\mathrm{e}^{\lambda_0^2t}$, which justifies the uniform $C^1$ convergence of $\mathrm{e}^{-\lambda_0^2t}y^\ast(\cdot,t)$.) Recalling \eqref{eq:intersection points for unique asymptotics when d=1}, we conclude that
\[
\frac{\tanh(\lambda_0(x_0+1))}{\lambda_0(x_0+1)}\to \frac{1}{2}\;\;\text{as}\;\; t\to-\infty\,.
\]
This implies that $x_0=1-o(1)$ as $t\to-\infty$ and hence, as $t\to-\infty$,
\[
\sin\overline\theta(t)\le\sin\overline\theta(t_0)=(1+\overline x(t_0))r_0^{-1}\sim (1+x_0)r_0^{-1}=\frac{B_0}{\sqrt{1+B_0^2}}\le B_0\sim \mathrm{e}^{\lambda_0^2t}
\]
as desired.
%
%But then
%\[
%\overline y(t_0)=y_{\overline\theta(t_0)}(\cos\overline\theta(t_0))\le (x_0+1)\left(\frac{A_0}{x_0+1}+\frac{1}{B_0}\right)\left[1-\sqrt{1-\left(\frac{1+\cos\overline\theta(t_0)}{(x_0+1)(\frac{A_0}{x_0+1}+\frac{1}{B_0})}\right)^2}\right]\,.
%\]
%(The right hand side is the largest value taken by the circular arc which satisfies the relations \eqref{eq:avoidance barrier argumet d=1} with equality.
\end{proof}

\subsection{Uniqueness}

Uniqueness may now be established using the avoidance principle, as in \cite[Proposition 3.5]{MR4668092}. Combined with Theorem \ref{thm:existence} and the asymptotics \eqref{eq:asymptotics}, this completes the proof of Theorem \ref{thm:convex ancient solutions}.

\bibliographystyle{plain}
\bibliography{bibliography}

\end{document}